\documentclass[11pt]{amsart}

\usepackage[T1]{fontenc}
\usepackage[utf8]{inputenc}

\usepackage{amsmath,amssymb,amsfonts,amsthm,mathtools}
\usepackage{mathrsfs}
\usepackage{enumitem}
\usepackage{tikz}
\usepackage[colorlinks=true,
            linkcolor=blue,
            citecolor=blue,
            urlcolor=blue]{hyperref}
\usepackage[nameinlink,noabbrev]{cleveref}
\usepackage{cite}

\numberwithin{equation}{section}

\newtheorem{theorem}{Theorem}[section]
\newtheorem{lemma}[theorem]{Lemma}
\newtheorem{proposition}[theorem]{Proposition}
\newtheorem{corollary}[theorem]{Corollary}

\theoremstyle{definition}
\newtheorem{definition}[theorem]{Definition}

\newtheorem{remark}[theorem]{Remark}

\title{Binary transformation groups and topological fields}

\author{Pavel S. Gevorgyan}
\address{%
\scalebox{0.94}[1]{Department of Mathematical Analysis named after Academician P.S. Novikov}\\
Moscow Pedagogical State University, Moscow, Russia}
\email{pgev@yandex.ru}


\begin{document}

\begin{abstract}
The notion of a semitransitive binary action of a group $G$ on a topological space is introduced. A duality theorem is proved, establishing a bijective correspondence between semitransitive distributive binary $G$-spaces and topological fields whose multiplicative group is isomorphic to $G$. This result yields an equivalence between the category of semitransitive distributive binary $G$-spaces and the category of topological fields with multiplicative group $G$.

As applications of the duality theorem, two important results are established. It is shown that a finite group can act semitransitively, distributively, and binarily only on finite sets whose cardinality is a power of a prime number. A complete characterization of those groups that can appear as multiplicative groups of topological fields is also obtained.
\end{abstract}

\maketitle

%--------------------------------

\section{Introduction}

Representations of groups by unary operations (permutations, bijections, transformations) on sets play an important role in many branches of mathematics and its applications. Similarly, representations of groups by binary operations (binary transformations) on sets occupy an important place in algebra and other areas of mathematics.

The set $C(X)$ of all continuous maps from a topological space $X$ to itself, endowed with the compact-open topology, is a topological monoid under composition of maps, with the identity map as the unit. All invertible elements of this monoid form the homeomorphism group $H(X)$ of the space $X$.

To construct the group of binary transformations of a topological space $X$, one should consider the set $C_2(X)$ of all continuous maps from $X^2$ to $X$ and define a natural group operation on it. There are various ways to define the product (composition) of two binary operations $\varphi$ and $\psi$ belonging to $C_2(X)$. However, the most natural products are those given by the formulas
\[
\varphi \psi (x,y) = \varphi(x, \psi (x,y)), \quad \varphi \psi (x,y) = \varphi(\psi (x,y),y), \quad x,y \in X,
\]
which are called the \textit{right} and \textit{left} multiplication of the operations $\varphi$ and $\psi$, respectively. These multiplications were introduced by Mann \cite{Man}. The set of all binary operations defined on a given set $X$ forms a monoid under left and right multiplication of operations. These two monoids are isomorphic, and the isomorphism is established by the correspondence $\varphi \to \varphi^*$, where $\varphi^*(x, y)=\varphi (y, x)$ \cite{Movsisyan}.

The set $C_2(X)$ of all continuous maps $f:X^2\to X$, endowed with the compact-open topology, is a topological monoid with identity $e(x,y)=y$ under right multiplication of operations \cite{Gev}. Let $H_2(X)$ be the group of all invertible elements of the monoid $C_2(X)$. The elements of the group $H_2(X)$ are called \emph{binary transformations}, and the group $H_2(X)$ itself is called the \emph{group of binary transformations} of the space $X$. The group $H(X)$ embeds (algebraically and topologically) into the group $H_2(X)$ \cite{Gev2}. Thus, the group of binary transformations $H_2(X)$ is a natural extension of the homeomorphism group $H(X)$.

Any topological group that acts effectively on a space $X$ is isomorphic to some topological subgroup of the group $H(X)$ \cite{Br}. Consequently, the homeomorphism group $H(X)$ is a universal object for topological groups acting on the space $X$. Similarly, the group of binary transformations $H_2(X)$ serves as a universal object for topological groups acting binarily on $X$.

The foundation of the theory of binary $G$-spaces, or groups of binary transformations, was initiated in \cite{Gev}. This theory was further developed in \cite{Gev2,Gev-Iliadis,Gev-Naz,{Gev-3},Gev-Quitzeh,Gev-Quitzeh-1,vestnikmgu}. In the theory of binary $G$-spaces, a special role is played by distributive binary $G$-spaces, because important concepts and results of equivariant topology extend to them \cite{Gev-Naz,Gev-3,{Gev-Quitzeh-1}}. In particular, a complete classification of transitive binary actions has been obtained for the class of distributive binary $G$-spaces \cite{vestnikmgu}.

In the present paper, we introduce the notion of a semitransitive binary $G$-space and solve the problem of classifying semitransitive binary actions in the class of distributive binary $G$-spaces. It is proved that every semitransitive distributive binary $G$-space admits a structure of a topological field compatible with the given binary action (Theorem \ref{th-main}). Thus, we establish a one-to-one correspondence between semitransitive distributive binary $G$-spaces and topological fields with the multiplicative group $G$ (the duality theorem \ref{th-classification}). This correspondence is compatible with isomorphisms of the corresponding objects. Hence, the category of semitransitive distributive binary $G$-spaces is equivalent to the category of topological fields whose multiplicative group is isomorphic to $G$. It should be noted that the proof of Theorem~\ref{th-main} relies on a number of ideas presented in \cite{belousov}, which is devoted to the study of distributive systems of operations.

Theorem~\ref{th-classification}, like many duality theorems, has important applications. One of them concerns the description of possible underlying sets of semitransitive distributive binary actions of finite groups: if there exists a semitransitive distributive binary action of a finite group $G$ on a space $X$, then the cardinality of this space is a power of a prime number (Corollary~\ref{cor-pn}).

Another important application of Theorem~\ref{th-classification} is related to the problem of describing multiplicative groups of topological fields. This problem has deep algebraic roots. As early as 1960, A.~I.~Maltsev formulated the hypothesis that it is impossible to describe multiplicative groups of fields by means of first-order formulas. This hypothesis was subsequently confirmed. Nevertheless, the problem is solvable using second-order formulas. The following assertion holds: a topological group $G$ is the multiplicative group of some topological field if and only if it admits a binary representation possessing the properties of semitransitivity and distributivity (Theorem~\ref{th-multgr}).

%----------------------------------

\section{Binary \texorpdfstring{$G$}{G}-spaces and Transitivity}

In this section, we recall the necessary definitions and present results from biequivariant topology \cite{Gev,Gev2,Gev-Iliadis,Gev-Naz,Gev-3,vestnikmgu} that will be used in what follows.

Let $X$ be a topological space, and let $C_2(X)$ be the set of all continuous binary operations on $X$. On the set $C_2(X)$, define the multiplication of two elements by the formula
\begin{equation}\label{eq-operation}
\varphi \psi (x,y) = \varphi(x, \psi (x,y)),
\end{equation}
where $\varphi, \psi \in C_2(X)$ and $x,y \in X$.

\begin{proposition}[\cite{Gev}]
The space $C_2(X)$, endowed with the compact-open topology, is a topological monoid under the operation~\eqref{eq-operation}, whose unit is \( e(x,y) = y \).
\end{proposition}

The invertible elements of the monoid $C_2(X)$ are called \emph{binary transformations} of the space $X$. Denote by $H_2(X)$ the group of all invertible elements of the monoid $C_2(X)$, that is, the group of binary transformations of the space $X$.

The following assertion holds \cite{Gev2}.

\begin{theorem}\label{th-00}
Let $f\in H_2(X)$ be a binary transformation of the space $X$. Then, for every $a\in X$, the map $f_a : X \to X$ defined by
\begin{equation}
f_a(x) = f(a, x)
\end{equation}
is a homeomorphism of the space $X$.
\end{theorem}

A binary operation $\varphi \in C_2(X)$ is called a continuous \emph{left quasigroup operation} on $X$ if the following condition holds: for any elements $a$ and $b$ in $X$, the equation $\varphi(a,x)=b$ has a unique solution $x \in X$.

Theorem \ref{th-00} immediately implies the following assertion.

\begin{proposition}
Let $f\in H_2(X)$ be a binary transformation of the space $X$. Then $f$ is a left quasigroup operation on $X$.
\end{proposition}

Thus, all elements of the group $H_2(X)$ are continuous left quasigroup operations on $X$.

Let $G$ be a topological group. A continuous map
\(
\alpha : G \times X^2 \to X
\)
is called a \textit{binary action} of the topological group $G$ on the space $X$ if the following conditions hold:

1) $\alpha (gh, x,y)=\alpha(g, x, \alpha(h,x,y))$;

2) $\alpha (e, x,y)=y$;

\noindent
for all $g,h\in G$ and $x,y\in X$, where $e$ is the identity element of the group $G$.

Using the notation $\alpha(g, x, y) = g(x,y)$, conditions 1) and 2) can be rewritten as
\[
gh(x,y) = g(x, h(x,y)), \qquad e(x,y) = y.
\]

A space $X$ equipped with a fixed binary action of the group $G$, or the triple $(G,X,\alpha)$, is called a \textit{binary $G$-space}.

It is easy to verify that any topological group $G$ acts binarily on itself by means of the \emph{conjugate left translation}
\begin{equation}\label{eq-slsdvig}
g(x,y) = x g x^{-1} y,
\end{equation}
where $g, x, y \in G$.

For arbitrary $g \in G$, consider the continuous binary operation $\alpha_g : X^2 \to X$ defined by
\[
\alpha_g(x,y) = \alpha(g,x,y).
\]
The map $g \mapsto \alpha_g$ defines a homomorphism
\[
\theta : G \to H_2(X)
\]
from the group $G$ to the group $H_2(X)$ of all binary transformations of the space $X$.

The kernel of this homomorphism $\theta$ is called the \emph{kernel of the binary action $\alpha$} and is denoted by ${\rm Ker}\,\alpha$. Thus,
\[
{\rm Ker}\,\alpha = \{g \in G \mid g(x,y) = y \ \text{for all } x,y \in X\}.
\]
This set is a closed normal subgroup of the group $G$.

The binary action $\alpha$ is called \emph{effective} if ${\rm Ker}\,\alpha = \{e\}$. Since any binary action of the group $G$ on $X$ with kernel ${\rm Ker}\,\alpha$ induces an effective binary action of the group $G/{\rm Ker}\,\alpha$ on $X$, we will henceforth consider only effective binary actions.

A \emph{binary representation} of a topological group $G$ is its continuous homomorphism into the group $H_2(X)$, where $X$ is some topological space. Since every binary action of a group $G$ on a space $X$ generates a binary representation $G \to H_2(X)$ and, conversely, every such representation defines a binary action on $X$, a binary representation of the group $G$ may be regarded as its binary action on the corresponding space $X$.

Let $(G,X,\alpha)$ and $(G,Y,\beta)$ be binary $G$-spaces. A continuous map $f: X \to Y$ is called \emph{biequivariant} if
\[
f(\alpha(g,x,y)) = \beta(g,f(x),f(y)),
\]
or, equivalently,
\[
f(g(x,y)) = g(f(x),f(y)),
\]
for all $g \in G$ and $x,y \in X$.

A biequivariant map $f: X \to Y$ that is a homeomorphism is called an \emph{equivalence} of binary $G$-spaces, or a \emph{biequimorphism}. All binary $G$-spaces together with biequivariant maps form a category.

A binary $G$-space $(G,X,\alpha)$ is called \emph{distributive} if
\begin{equation*}\label{eq1-1}
g(x,h(y,z)) = h(g(x,y), g(x,z))
\end{equation*}
for all $g,h \in G$ and $x,y,z \in X$. In this case, $\alpha$ is called a \emph{distributive binary action} of the group $G$ on $X$.

The binary action of a group $G$ on itself by means of the conjugate left translation
\[
g(x,y) = x g x^{-1} y
\]
is a distributive binary action.

For arbitrary $g \in G$ and $x \in X$, consider the map $g_x : X \to X$ defined by
\[
g_x(y) = g(x,y), \quad y \in X.
\]
The following criterion for distributivity of a binary action of a group $G$ on $X$ holds.

\begin{theorem}[\cite{vestnikmgu}]\label{th-criteria}
A binary $G$-space $X$ is distributive if and only if one of the following conditions holds:

{\rm 1)} the map $g_x : X \to X$ is a biequivariant homeomorphism for all $x \in X$;

{\rm 2)} the equality $g_x h_y = h_{g_x(y)} g_x$ holds for all $g,h \in G$ and $x,y \in X$.
\end{theorem}

The set
\[
G_{(x,y)} = \{g \in G \mid g(x,y) = y\},
\]
where $(x,y)$ is a fixed pair of points of the binary $G$-space $X$, is a closed subgroup of the group~$G$. This subgroup is called the \emph{stationary subgroup} of the pair $(x,y)$. The subgroup $G_{(x,x)}$ is called the stationary subgroup of the point $x$.

In general, the stationary subgroup $G_{(x,y)}$ is not a normal subgroup of the group $G$. However, the following assertion holds.

\begin{proposition}[\cite{vestnikmgu}]
Let $X$ be a distributive binary $G$-space, and let $x \in X$ be an arbitrary point. Then the stationary subgroup $G_{(x,x)}$ is a closed normal subgroup of the group $G$.
\end{proposition}

A binary action of a group $G$ on $X$ is called \emph{transitive} if $G(x,x)=X$ for every $x \in X$. In this case, $X$ is called a \emph{transitive binary $G$-space}~\cite{vestnikmgu}. In fact, transitivity of a binary action implies that $G(x,y)=X$ for all $x,y \in X$.

Let $H$ be a closed normal subgroup of the group $G$. The map
\[
\alpha : G \times G/H \times G/H \to G/H,
\]
defined by
\begin{equation}\label{eq-g(kH,lH)}
g(kH,lH) = (kgk^{-1}l)H,
\end{equation}
where $g,k,l \in G$ and $g(kH,lH) = \alpha(g,kH,lH)$, is a transitive distributive binary action of the group $G$ on $G/H$.

The following theorem shows that, in the class of all distributive binary $G$-spaces, the spaces $G/H$ with the binary action~\eqref{eq-g(kH,lH)} exhaust all transitive binary $G$-spaces.

\begin{theorem}[\cite{vestnikmgu}]\label{th-2}
Every transitive distributive binary $G$-space $X$, where $G$ is a compact group, is biequivariantly homeomorphic to the coset space $G/H$ for some closed normal subgroup $H$ of $G$, equipped with the binary action~\eqref{eq-g(kH,lH)}.
\end{theorem}

It follows from this theorem that, in the class of all effective distributive binary $G$-spaces, there exists only one transitive binary $G$-space. More precisely, the following assertion holds.

\begin{theorem}[\cite{vestnikmgu}]\label{th-3}
Let $G$ be a compact group. Then any effective binary $G$-space that is both transitive and distributive is biequivariantly homeomorphic to the topological group $G$ with the binary action given by the conjugate left translation~\eqref{eq-slsdvig}.
\end{theorem}

%--------------------------------

\section{Semitransitive binary \texorpdfstring{$G$}{G}-spaces and topological fields}

\begin{definition}\label{def:semi-trans}
A binary $G$-space $X$ is called \emph{semitransitive} if the following conditions hold:

{\rm 1)} $G(x,x)=x$ for any $x\in X$ and

{\rm 2)} $G(x,y)=X \setminus \{x\}$ for any distinct points $x,y \in X$.
\end{definition}

\begin{proposition}\label{ex-pole}
Let $F$ be a topological field, and let $G$ be the multiplicative group of this field. Then the map
\[
\alpha : G \times F^2 \to F,
\]
defined by
\begin{equation}\label{eq-pole}
p(a, b) = (1 - p)a + p b,
\end{equation}
where $p(a, b) = \alpha(p, a, b)$, $p \in G$, $a, b \in F$, is a semitransitive distributive binary action of the group $G$ on $F$.
\end{proposition}

\begin{proof}
We first prove that $\alpha$ is a binary action. Indeed,

\begin{enumerate}[label=\textup{\arabic*)}]
\item $\begin{aligned}[t]
p(a,q(a,b))
&= (1-p)a + p\bigl((1-q)a+qb\bigr) \\
&= a-pa+p(1-q)a+pqb \\
&= a-pqa+pqb \\
&= (1-pq)a+pqb \\
&= pq(a,b);
\end{aligned}$

\item $1(a,b) = (1-1)a+1b=b.$
\end{enumerate}

To prove distributivity of this binary action, note that
\begin{align*}
p(a,q(b,c))
&= (1-p)a + p\bigl((1-q)b+qc\bigr) \\
&= (1-p)a + p(1-q)b + pqc, \\[5pt]
q\bigl(p(a,b),p(a,c)\bigr)
&= (1-q)\bigl((1-p)a+pb\bigr)
 + q\bigl((1-p)a+pc\bigr) \\
&= (1-q)(1-p)a + (1-q)pb
 + q(1-p)a + qpc \\
&= (1-p)(1-q+q)a + p(1-q)b + pqc \\
&= (1-p)a + p(1-q)b + pqc.
\end{align*}
Thus,
\[
p(a, q(b, c)) = q(p(a, b), p(a, c)),
\]
which means that the binary action $\alpha$ is distributive.

Now we prove that $\alpha$ satisfies the conditions of Definition~\ref{def:semi-trans} of semitransitivity.

1) For all $p \in G$ and $a \in F$ we have
    \[
    p(a, a) = (1 - p)a + pa = a.
    \]

2) Let $a, b \in F$, $a \ne b$. For any $c \in F \setminus \{a\}$ the equation
    \[
    (1 - p)a + p b = c
    \]
    has a unique solution for $b$, namely:
   \[
    b = \frac{1}{p}\left(c - (1 - p)a\right).
    \]
Therefore, $G(a, b) = F \setminus \{a\}$. 
\end{proof}

Thus, there exists a natural semitransitive distributive binary action of the multiplicative group of a topological field on the field itself, given by formula~\eqref{eq-pole}.

We will now show that there are no other semitransitive distributive binary $G$-spaces (see Theorem~\ref{th-main}).

\begin{proposition}\label{prop-1,2}
Let $X$ be a semitransitive distributive binary $G$-space. Then for any distinct points $x,y \in X$ the isotropy subgroup $G_{(x,y)}$ is trivial.
\end{proposition}

\begin{proof}
Consider arbitrary distinct points $x,y \in X$ and suppose $g \in G_{(x,y)}$, i.e., $g(x,y)=y$. We prove that then $g(x,z)=z$ for any $z \in X$, $z \ne y$.

Indeed, by semitransitivity there exists $h \in G$ such that $h(x,y)=z$. Then
\[
g(x,z)=g(x,h(x,y))=h(g(x,x),g(x,y))=h(x,y)=z.
\]

Thus, $g(x,z)=z$ for all $z \in X$. Consequently, $g \in \mathrm{Ker}\,\alpha = \{e\}$, i.e., $g=e$.
\end{proof}

\begin{proposition}\label{prop-abel}
Let $X$ be a semitransitive distributive binary $G$-space. Then $G$ is an abelian group.
\end{proposition}

\begin{proof}
For arbitrary $g,h\in G$ and $x,y\in X$ we have:
$$gh(x,y)= g(x, h(x,y)) = h(g(x,x), g(x,y)) = h(x, g(x,y)) = hg(x,y).$$
Hence, by Proposition~\ref{prop-1,2}, it follows that $gh=hg$.
\end{proof}

The following lemma plays an important role in the subsequent constructions.

\begin{lemma}\label{prop-1dist}
Let $X$ be a semitransitive distributive binary $G$-space. Then for any $g \in G$ there exists a unique element $\bar{g} \in G$ such that
\begin{equation}\label{eq-mn}
g(\bar{g}(x,y),x) = y
\end{equation}
for all $x,y \in X$.
\end{lemma}

\begin{proof}
Let $x_0,y_0$ be some fixed elements of $X$. Consider the equation
\begin{equation}\label{eq-uravn}
g(t,x_0)=y_0
\end{equation}
for $t \in X$.

If $x_0=y_0$, then equation \eqref{eq-uravn} has the unique solution $t=x_0$.

Suppose $x_0 \ne y_0$. Since $g(y_0,x_0)\ne x_0$, by Proposition~\ref{prop-1,2} there exists a unique element $\bar{g}\in G$ such that
\[
\bar{g}(x_0, g(y_0,x_0))=y_0.
\]
Applying distributivity and semitransitivity of the binary action, this equality can be rewritten as
\[
g(\bar{g}(x_0,y_0),\bar{g}(x_0,x_0)) = g(\bar{g}(x_0,y_0),x_0)=y_0.
\]
This means that $t=\bar{g}(x_0,y_0)$ is a solution of equation \eqref{eq-uravn}.

Let us prove that this solution is unique. Suppose
\[
g(t_0,x_0)=y_0 \quad \text{and} \quad g(t_1,x_0)=y_0
\]
with $t_0 \ne t_1 \ne x_0$. Then there exists an element $k\in G$ such that $k(x_0,t_0)=t_1$. Note that
\[
y_0=g(t_1,x_0)=g(k(x_0,t_0),x_0)=k(x_0,g(t_0,x_0))=k(x_0,y_0),
\]
hence, by Proposition~\ref{prop-1,2}, we obtain $k=e$. But then $t_1=k(x_0,t_0)=e(x_0,t_0)=t_0$. This contradiction proves the uniqueness of the solution.

Thus, equality \eqref{eq-mn} holds for $x_0,y_0\in X$, $x_0\ne y_0$:
\begin{equation}\label{eq-x0y0}
 g(\bar{g}(x_0,y_0),x_0)=y_0.
\end{equation}

Now we will prove that the element $\bar{g}\in G$ is the desired one, i.e., equality \eqref{eq-mn} holds for all $x,y\in X$.

First, note that
\[
 g(\bar{g}(x,y_0),x)=y_0
\]
for all $x \in X$, $x\ne y_0$. Indeed, by semitransitivity there exists $k \in G$ such that
\[
k(y_0,x_0)=x.
\]
Then, taking \eqref{eq-x0y0} into account, we have
\begin{align*}
g(\bar{g}(x,y_0),x)
&= g\bigl(\bar{g}(k(y_0,x_0),y_0),\,k(y_0,x_0)\bigr) \\
&= g\bigl(k(y_0,\bar{g}(x_0,y_0)),\,k(y_0,x_0)\bigr) \\
&= k\bigl(y_0,g(\bar{g}(x_0,y_0),x_0)\bigr) \\
&= k(y_0,y_0) = y_0 .
\end{align*}

It remains to prove that
\[
 g(\bar{g}(x_0,y),x_0)=y
\]
for all $y \in X$, $y\ne x_0$. Indeed, since $x_0\ne y_0\ne y$, by semitransitivity there exists $k \in G$ such that $k(x_0,y_0)=y$. Then
\begin{align*}
g(\bar{g}(x_0,y),x_0)
&= g\bigl(\bar{g}(x_0,k(x_0,y_0)),x_0\bigr) \\
&= g\bigl(k(x_0,\bar{g}(x_0,y_0)),x_0\bigr) \\
&= k\bigl(x_0,g(\bar{g}(x_0,y_0),x_0)\bigr) \\
&= k(x_0,y_0) = y .
\end{align*}

The lemma is proved.
\end{proof}

\begin{corollary}\label{cor-001}
Let $X$ be a semitransitive distributive binary $G$-space. Then for any $g \in G$ and any $x,y \in X$ the following equality holds:
\begin{equation}\label{eq-111}
g(x,y) = \bar{g}^{-1}(y,x),
\end{equation} 
where the element $\bar{g} \in G$ is defined by equality~\eqref{eq-mn}.
\end{corollary}

\begin{proof}
Indeed, by Lemma~\ref{prop-1dist} we have $g(\bar{g}(y,x),y) = x$, or equivalently,
\[
\bar{g}(y, g(x,y)) = x.
\] 
The required equality \eqref{eq-111} follows immediately.
\end{proof}

In what follows we will use the notation
\begin{equation*}
\hat{g} = \bar{g}^{-1}
\end{equation*}
and, when necessary, refer to the identity
\begin{equation}\label{eq-11100}
g(x,y) = \hat{g}(y,x).
\end{equation}

\begin{corollary}\label{cor-002}
Let $X$ be a semitransitive distributive binary $G$-space. Then for any $g \in G$ there exists a unique element $\tilde{g} \in G$ such that
\begin{equation*}
g(\tilde{g}(y,x),x) = y
\end{equation*} 
for all $x,y \in X$.
\end{corollary}

\begin{proof}
By Lemma~\ref{prop-1dist}, for a given $g \in G$ there exists a unique element $\bar{g}$ such that $g(\bar{g}(x,y),x) = y$. Define
\[
\tilde{g} = (\bar{\bar{g}})^{-1}.
\]
Then, using equality~\eqref{eq-111}, we obtain
\[
\bar{g}(x,y) = \tilde{g}(y,x).
\]
Therefore,
\(
g(\tilde{g}(y,x),x) = y.
\)
\end{proof}

It follows from the last corollary that $g:X^2 \to X$ has a \emph{left inverse operation} $\tilde{g}:X^2 \to X$. Moreover, from the definition of the binary action, it follows that the operation $g:X^2 \to X$ also has a right inverse operation. Thus, the following statement holds.

\begin{proposition}
Let $X$ be a semitransitive distributive binary $G$-space. Then any element $g\in G$, $g\neq e$, regarded as a map $g:X^2\to X$, is a continuous quasigroup operation on $X$.
\end{proposition}

In semitransitive distributive binary $G$-spaces the right distributivity law also holds.

\begin{proposition}
Let $X$ be a semitransitive distributive binary $G$-space. Then the following equality holds:
\begin{equation}\label{eq-0005}
g(h(x,y),z) = h(g(x,z), g(y,z))
\end{equation}
for all $g,h\in G$ and $x,y,z\in X$.
\end{proposition}

\begin{proof}
Indeed, by~\eqref{eq-11100} we have
\[
g(h(x,y),z) = \hat{g}(z, \hat{h}(y,x)) = \hat{h}(\hat{g}(z,y), \hat{g}(z,x)) = h(g(x,z), g(y,z)),
\]
which completes the proof.
\end{proof}

As the following lemma shows, in semitransitive distributive binary $G$-spaces a stronger identity than the right distributivity formula~\eqref{eq-0005} holds.

\begin{lemma}
Let $X$ be a semitransitive distributive binary $G$-space. Then
\begin{equation}\label{eq-001}
g(h(x,y),h(z,t)) = h(g(x,z),g(y,t))
\end{equation}
for all $g,h \in G$ and $x,y,z,t \in X$.
\end{lemma}

\begin{proof}
Consider the map $\varphi : X \to X$ defined by
\begin{equation*}
\varphi(t) = g^{-1}\bigl(y,\, h^{-1}(g(x,z),\, g(h(x,y),h(z,t)))\bigr),
\end{equation*}
where $t \in X$ and $x,y,z$ are fixed elements of $X$. From this definition it follows directly that $\varphi$ is a biequivariant homeomorphism of $X$. To prove~\eqref{eq-001}, it suffices to show that $\varphi$ is the identity map, i.e., $\varphi(t)=t$ for all $t \in X$. 

Suppose $y \ne z$. Note that
\begin{align*}
\varphi(y)
&= g^{-1}\Bigl(y,
    h^{-1}\bigl(g(x,z),g(h(x,y),h(z,y))\bigr)\Bigr) \\
&= g^{-1}\Bigl(y,
    h^{-1}\bigl(g(x,z),h(g(x,z),y)\bigr)\Bigr) \\
&= g^{-1}(y,y)
 = y .
\end{align*}
Similarly,
\begin{align*}
\varphi(z)
&= g^{-1}\Bigl(y,
    h^{-1}\bigl(g(x,z),g(h(x,y),h(z,z))\bigr)\Bigr) \\
&= g^{-1}\Bigl(y,
    h^{-1}\bigl(g(x,z),g(h(x,y),z)\bigr)\Bigr) \\
&= g^{-1}\Bigl(y,
    h^{-1}\bigl(g(x,z),h(g(x,z),g(y,z))\bigr)\Bigr) \\
&= g^{-1}\bigl(y,g(y,z)\bigr)
 = z .
\end{align*}

Now consider an arbitrary $t \in X$, $t \ne y$, $t \ne z$. Then there exists $h \in G$ such that $t = h(y,z)$. Therefore,
\[
\varphi(t) = \varphi(h(y,z)) = h(\varphi(y),\varphi(z)) = h(y,z) = t.
\]

Thus, equality~\eqref{eq-001} is proved in the case $y \ne z$.

Now suppose $y = z$. In this case we need to prove
\begin{equation}\label{eq-003}
g(h(x,y),h(y,t)) = h(g(x,y), g(y,t)).
\end{equation}

Since $h(x,y) \ne y$, using the already proved formula we obtain
\[
h^{-1}(g(x,y), g(h(x,y),h(y,t))) = g\bigl(h^{-1}(x,h(x,y)),\, h^{-1}(y,h(y,t))\bigr) = g(y,t),
\]
from which the required equality~\eqref{eq-003} follows directly.
\end{proof}

The results obtained above, together with several ideas from~\cite{belousov} on distributive systems of operations, allow us to formulate and prove the following theorem, which establishes a correspondence between semitransitive distributive binary $G$-spaces and topological fields.

\begin{theorem}\label{th-main}
Let $(G, X, \alpha)$ be a semitransitive distributive binary $G$-space. Then $X$ admits a topological field structure such that $G$ is isomorphic to the multiplicative group of this field and $\alpha$ is the binary action given by formula~\eqref{eq-pole}.
\end{theorem}

\begin{proof}
Fix two arbitrary distinct points $x_0$ and $x_1$ of $X$. Denote $x_0 = 0$ and $x_1 = 1$. Consider the action $\alpha_0$ of the group $G$ on $X$ defined by
\[
\alpha_0(g,x) = g(0,x).
\]
By the semitransitivity of the binary action, the space $X$ has two orbits under $\alpha_0$: $\{0\}$ and $X\setminus\{0\}$. By Proposition~\ref{prop-1,2}, the map
\[
i: G \to X\setminus\{0\}, \qquad i(g) = g(0,1),
\]
is a bijection, and its inverse
\[
i^{-1}: X\setminus\{0\} \to G
\]
associates to each element $x \in X\setminus\{0\}$ the element $g_x = i^{-1}(x)$ of $G$ such that
\[
g_x(0,1) = x.
\]

We define a multiplication on $X$ as follows:
\begin{equation}\label{umnojenie}
\begin{cases}
xy = i(g_x g_y) = (g_x g_y)(0,1) = g_x(0,y), & x,y\in X,\ x\neq 0,\ y\neq 0,\\[2mm]
x0 = 0x = 0, & x\in X.
\end{cases}
\end{equation}

We now prove that with respect to the multiplication~\eqref{umnojenie}, the set $X\setminus\{0\}$ is a commutative group with identity $1$.

\emph{Associativity.}
For all $x,y,z \in X\setminus\{0\}$ we have
\begin{align*}
(xy)z
&= i\bigl((g_x g_y)g_z\bigr)
 = ((g_x g_y)g_z)(0,1) \\
&= (g_x(g_y g_z))(0,1)
 = i\bigl(g_x(g_y g_z)\bigr)
 = x(yz).
\end{align*}

\emph{Existence of identity.}
Since $e(0,1) = 1$, for any $x \in X\setminus\{0\}$ we have
\[
x \cdot 1 = i(g_x e) = g_x(0,1) = x, 
\qquad
1 \cdot x = i(e g_x) = g_x(0,1) = x.
\]

\emph{Existence of inverses.}
For any $x \in X\setminus\{0\}$, denote $x^{-1} = i(g_x^{-1}) = g_x^{-1}(0,1)$. Then
\[
x \cdot x^{-1} = i(g_x g_x^{-1}) = i(e) = 1,
\qquad
x^{-1} \cdot x = i(g_x^{-1} g_x) = i(e) = 1.
\]

\emph{Commutativity.}
By Proposition~\ref{prop-abel}, we have $g_x g_y = g_y g_x$ for all $x,y \in X\setminus\{0\}$. Therefore,
\[
xy = i(g_x g_y) = i(g_y g_x) = yx.
\]

\emph{Continuity.}
Multiplication~\eqref{umnojenie} and the map $x \mapsto x^{-1}$ are continuous because the binary action of the group $G$ on $X$ is continuous.

Now fix an arbitrary element $s \in G$ and define an addition operation $+$ on $X$ by the formula
\begin{equation}\label{eq-sum}
x + y = s\bigl(\tilde{s}(x,0),\, s^{-1}(0,y)\bigr),
\end{equation}
where $x,y \in X$ and $\tilde{s}$ is the left inverse of $s$ (see Corollary~\ref{cor-002}).

We first note that for any $g \in G$ the identity
\begin{equation}\label{eq-004}
g(x+y,\, z+t) = g(x,z) + g(y,t)
\end{equation}
holds for all $x,y,z,t \in X$. Indeed, by~\eqref{eq-001} we have
\begin{align*}
g(x+y,\, z+t)
&= g\Bigl(
    s\bigl(\tilde{s}(x,0), s^{-1}(0,y)\bigr),\,
    s\bigl(\tilde{s}(z,0), s^{-1}(0,t)\bigr)
  \Bigr) \\
&= s\Bigl(
    g\bigl(\tilde{s}(x,0),\tilde{s}(z,0)\bigr),\,
    g\bigl(s^{-1}(0,y), s^{-1}(0,t)\bigr)
  \Bigr) \\
&= s\Bigl(
    \tilde{s}\bigl(g(x,z),0\bigr),\,
    s^{-1}\bigl(0, g(y,t)\bigr)
  \Bigr) \\
&= g(x,z) + g(y,t).
\end{align*}

We now prove that $X$ is an abelian group under the addition~\eqref{eq-sum}.

\emph{Associativity.}
This follows from formulas~\eqref{eq-sum} and~\eqref{eq-004}:
\begin{align*}
x+(y+z)
&= s\bigl(\tilde{s}(x,0),\, s^{-1}(0,y+z)\bigr) \\
&= s\bigl(
    \tilde{s}(x,0)+0,\,
    s^{-1}(0,y)+s^{-1}(0,z)
  \bigr) \\
&= s\bigl(\tilde{s}(x,0), s^{-1}(0,y)\bigr)
 + s\bigl(0, s^{-1}(0,z)\bigr) \\
&= (x+y)+z .
\end{align*}

\emph{Existence of zero.}
Note that $0$ is the zero element:
\[
x+0 = s\bigl(\tilde{s}(x,0), s^{-1}(0,0)\bigr) = s\bigl(\tilde{s}(x,0),0\bigr) = x,
\]
\[
0+y = s\bigl(\tilde{s}(0,0), s^{-1}(0,y)\bigr) = s\bigl(0, s^{-1}(0,y)\bigr) = y.
\]

\emph{Existence of additive inverses.}
Let $x$ be an arbitrary element of $X$. We show that the element
\[
- x = s\bigl(0,\, s^{-1}(\tilde{s}(x,0),0)\bigr)
\]
is the additive inverse of $x$. Indeed,
\begin{align*}
x+(-x)
&= s\Bigl(
    \tilde{s}(x,0),\,
    s^{-1}\bigl(0,s(0,s^{-1}(\tilde{s}(x,0),0))\bigr)
  \Bigr) \\
&= s\bigl(
    \tilde{s}(x,0),\,
    s^{-1}(\tilde{s}(x,0),0)
  \bigr) \\
&= e\bigl(\tilde{s}(x,0),0\bigr) = 0.
\end{align*}

\emph{Commutativity.}
Note that the binary action of the group $G$ on $X$ and the addition operation~\eqref{eq-sum} are related by the formula
\begin{equation}\label{eq-005}
g(x, y) = g(x,0) + g(0,y)
\end{equation}
for all $g \in G$ and $x,y \in X$. Indeed, by~\eqref{eq-004}, we have
\[
g(x, y) = g(x+0, 0+y) = g(x,0) + g(0,y).
\]

Now, using~\eqref{eq-sum},~\eqref{eq-005} and~\eqref{eq-11100} (see Corollary~\ref{cor-001}), we obtain
\begin{align*}
x+y
&= s\bigl(\tilde{s}(x,0),\,s^{-1}(0,y)\bigr) \\
&= \hat{s}\bigl(s^{-1}(0,y),\,\tilde{s}(x,0)\bigr) \\
&= \hat{s}\bigl(s^{-1}(0,y),0\bigr)
 + \hat{s}\bigl(0,\tilde{s}(x,0)\bigr) \\
&= s\bigl(0,s^{-1}(0,y)\bigr)
 + s\bigl(\tilde{s}(x,0),0\bigr) \\
&= y+x .
\end{align*}

\emph{Continuity.}
The addition operation~\eqref{eq-sum} and the map $x \mapsto -x$ are continuous because the binary action of $G$ on $X$ is continuous.

We now prove the distributivity property. By~\eqref{umnojenie} and~\eqref{eq-004}, we have
\[
x(y+z) = g_x(0, y+z) = g_x(0+0, y+z) = g_x(0,y) + g_x(0,z) = xy + xz.
\]

Thus, $X$ is a topological field with respect to the operations defined by formulas~\eqref{umnojenie} and~\eqref{eq-sum}.

It remains to show that the binary action of the group $G$ on $X$ coincides with the natural action~\eqref{eq-pole} of the multiplicative group of the field on $X$.
From identity~\eqref{eq-005} it follows that
\[
x = g(x,x) = g(x,0) + g(0,x),
\]
hence
\[
g(x,0) = x - g(0,x).
\]
Therefore,
\[
g(x,y) = g(x,0) + g(0,y) = x - g(0,x) + g(0,y).
\]
Denote $p = i(g) = g(0,1) \in X\setminus\{0\}$. Then $g(0,x) = px$ and $g(0,y) = py$. Consequently,
\[
g(x,y) = p(x,y) = x - px + py = (1-p)x + py.
\]

The theorem is proved.
\end{proof}

The topological field constructed in the proof of the previous theorem depends on three parameters \(x_0, x_1, s\), where \(x_0, x_1 \in X\), \(x_0 \neq x_1\), and \(s \in G\). Denote this field by
\[
F = F(X, x_0, x_1, s).
\]
The following statement holds.

\begin{proposition}
Let \(x_0 \neq x_1\) and \(x_0' \neq x_1'\) be arbitrary pairs of distinct points in a semitransitive distributive binary \(G\)-space \(X\), and let \(s, s'\) be arbitrary elements of the group \(G\). Then the topological fields \(F = (X, x_0, x_1, s)\) and \(F' = (X, x_0', x_1', s')\) are isomorphic.
\end{proposition}

\begin{proof}
By Theorem~\ref{th-main}, the binary action of the group \(G\) on \(X\) is expressed in terms of the addition <<\(+\)>> and multiplication <<\(\cdot\)>> of the topological field \(F = (X, x_0, x_1, s)\) by the formula
\[
g(x,y) = (1-p)\cdot x + p\cdot y,
\]
where \(p = g(0,1)\), and \(x_0 = 0\) and \(x_1 = 1\) play the roles of the zero and identity of the field.

On the other hand, the same binary action of the group \(G\) on \(X\) is also expressed in terms of the addition <<\(\oplus\)>> and multiplication <<\(\circ\)>> of the topological field \(F' = (X, x_0', x_1', s')\) by the formula
\[
g(x,y) = (1' \ominus p')\circ x \oplus p'\circ y,
\]
where \(p' = g(0',1')\), and \(x_0' = 0'\) and \(x_1' = 1'\) are the zero and identity of the field \(F'\), respectively.

Thus, for any \(p \in X\) there exists a unique \(p' \in X\) such that the equality
\begin{equation}\label{varphi(p)}
(1-p)\cdot x + p\cdot y = (1' \ominus p')\circ x \oplus p'\circ y
\end{equation}
holds for all \(x, y \in X\).

Define a map \(\varphi : F \to F'\) by \(\varphi(p) = p'\), and prove that \(\varphi\) is an isomorphism of the fields \(F\) and \(F'\).

Substituting \(x = 0'\) and \(y = 1'\) into \eqref{varphi(p)}, we obtain
\[
p' = (1-p)\cdot 0' + p\cdot 1' = p\cdot (1'-0') + 0'.
\]
Thus, \(\varphi\) is given by the affine linear map
\begin{equation}\label{eq-varphi}
\varphi(p) = p\cdot (1'-0') + 0'.
\end{equation}
Consequently, \(\varphi\) is a homeomorphism of the space \(X\) that preserves the additive and multiplicative neutral elements of the fields \(F\) and \(F'\):
\[
\varphi(0) = 0' \quad \text{and} \quad \varphi(1) = 1'.
\]

We now prove that \(\varphi\) also preserves the algebraic structures of these fields.

Consider the following two identities, which follow directly from equality \eqref{varphi(p)} by substituting \(y = 0'\) and \(x = 0'\), respectively:
\begin{equation}\label{eq-y=0'}
(1-p)\cdot x + p\cdot 0' = (1' \ominus p')\circ x,
\end{equation}
\begin{equation}\label{eq-x=0'}
(1-p)\cdot 0' + p\cdot y = p'\circ y.
\end{equation}

Let \(x, y \in X\) be arbitrary elements. They can be represented as
\begin{equation}\label{eq-xy}
x = (1' \ominus p')\circ z, \qquad y = p' \circ t,
\end{equation}
where \(z, t \in F'\) and \(p' \in F'\) is a fixed element different from \(0'\) and \(1'\).

Using equalities \eqref{varphi(p)}, \eqref{eq-y=0'}, \eqref{eq-x=0'} and representation \eqref{eq-xy}, we obtain
\begin{align*}
x+y
&= \bigl((1' \ominus p')\circ z\bigr) + (p'\circ t) \\
&= (1-p)\cdot z + p\cdot 0'
 + (1-p)\cdot 0' + p\cdot t \\
&= (1-p)\cdot z + p\cdot t + 0' \\
&= \bigl((1' \ominus p')\circ z \oplus p'\circ t\bigr) + 0' \\
&= (x\oplus y)+0' .
\end{align*}
Hence we obtain the formula
\begin{equation}\label{eq-xoplusy}
x \oplus y = x + y - 0'.
\end{equation}

Using equalities \eqref{eq-varphi} and \eqref{eq-xoplusy}, it is easy to show that \(\varphi\) is a homomorphism of the additive groups \((F, +)\) and \((F', \oplus)\):
\begin{align*}
\varphi(x)\oplus\varphi(y)
&= \varphi(x)+\varphi(y)-0' \\
&= \bigl(x(1'-0')+0'\bigr)
 + \bigl(y(1'-0')+0'\bigr) - 0' \\
&= (x+y)\cdot(1'-0')+0'
 = \varphi(x+y).
\end{align*}

Next, rewrite equality \eqref{eq-x=0'} as
\begin{equation}\label{eq-20}
\varphi(x) \circ y = x \cdot (y - 0') + 0',
\end{equation}
where \(x, y \in F\). Now from equalities \eqref{eq-varphi} and \eqref{eq-20} we directly obtain
\begin{align*}
\varphi(x)\circ\varphi(y)
&= x\cdot\bigl(\varphi(y)-0'\bigr)+0' \\
&= x\cdot\bigl(y\cdot(1'-0')+0'-0'\bigr)+0' \\
&= x\cdot y\cdot(1'-0')+0'
 = \varphi(x\cdot y).
\end{align*}
Thus, \(\varphi\) is a homomorphism of the multiplicative groups \((F, \cdot)\) and \((F', \circ)\).

Therefore, \(\varphi\) is an isomorphism of the topological fields \(F\) and \(F'\).

The proposition is proved.
\end{proof}

From Theorem~\ref{th-main} and Proposition~\ref{ex-pole} the following important theorem follows directly.

\begin{theorem}[Duality theorem]\label{th-classification}
There exists a one-to-one correspondence between the following objects:

{\rm 1)} semitransitive distributive binary \(G\)-spaces;

{\rm 2)} topological fields whose multiplicative group is isomorphic to \(G\).
\end{theorem}

\begin{remark}
The correspondence established in the last theorem is compatible with isomorphisms of the corresponding objects. Thus, Theorem~\ref{th-classification} establishes an equivalence between the category of semitransitive distributive binary \(G\)-spaces and the category of topological fields whose multiplicative group is isomorphic to \(G\).
\end{remark}

Theorem~\ref{th-classification}, like any duality theorem, opens the way to a wide range of applications. One of them is related to the classification of semitransitive distributive binary actions of finite groups. In particular, from Theorem~\ref{th-classification}, taking into account the structure of finite fields, the following statement follows directly.

\begin{corollary}\label{cor-pn}
Let \(G\) be a finite group, and let \(X\) be a semitransitive distributive binary \(G\)-space. Then \(|X| = p^n\), where \(p\) is a prime number and \(n\) is a natural number.
\end{corollary}

Not every topological group admits a semitransitive distributive binary action on a topological space. By Theorem~\ref{th-classification}, the problem of describing such groups is equivalent to the problem of classifying groups that can be multiplicative groups of topological fields.

The classical problem of describing multiplicative groups of fields is one of the important and difficult problems in algebra (see~\cite{Movsisyan, Dicker, fuchs}). In 1960, at the Third Colloquium on General Algebra in Moscow, A. I. Maltsev conjectured that the class of multiplicative groups of fields cannot be characterized elementarily, i.e., by first-order formulas~\cite{ershov}. This conjecture was proved in~\cite{kogalovski} (see also~\cite{sabbagh}). However, multiplicative groups of fields can be characterized by second-order formulas or hyperidentities~\cite{Movsisyan}.

From Theorem~\ref{th-classification} the following characterization of multiplicative groups of topological fields follows directly.

\begin{theorem}\label{th-multgr}
A group \(G\) is the multiplicative group of some topological field if and only if there exists a semitransitive and distributive binary representation of \(G\).
\end{theorem}

%Bibliography

\end{document}